\documentclass[11pt]{amsart}
\usepackage{}
\usepackage{amssymb}
\theoremstyle{plain}
\newtheorem{Thm}{Theorem}

\newtheorem{Pro}[Thm]{Proposition}
\newtheorem{Lem}[Thm]{Lemma}

\begin{document}

\title[Lambda constant and ground states]
{Lambda constant and ground states of Perelman's W-functionals}

\author{Li Ma}
\address{L. Ma, Distinguished Professor, Department of mathematics \\
Henan Normal university \\
Xinxiang, 453007 \\
China} \email{lma@tsinghua.edu.cn}

\thanks{The research is partially supported by the National Natural Science
Foundation of China No. 11271111}

\begin{abstract}
In this paper, we consider the generalized lambda constant and the existence of ground states of the generalized Perelman's W-functional from a variational formulation.  One result is concerned with the estimation of the generalized $\lambda$ constant. The other results are about the existence of ground states of generalized $F$-functional and W-functional both on a complete non-compact Riemannian manifold $(M,g)$. Our main results are Theorems \ref{bound}, \ref{bound2}, and \ref{lions}. For the existence of the ground states we use Lions' concentration-compactness method.


\end{abstract}

 \maketitle

\section{introduction}
This paper has two parts. One is about the estimation of the generalized $\lambda$ constant (see \cite{P}) and the other is about the existence of ground states of the generalized $F$-functional and W-functional both on a complete non-compact Riemannian manifold. The generalized functionals are defined by replacing the scalar curvature by a new function $V$ in $M$. This topic follows our works \cite{m} \cite{ma} and \cite{msz}. Our main results are Theorems \ref{bound}, \ref{bound2},\ref{lambda2} and \ref{lions}. We mention one result below: Assume that $(M,g)$ is a complete non-compact manifold with bounded Ricci curvature and with polynomial volume growth. Then the lambda constant of $g$ can not be positive provided the function $V$ has a special structure. For the existence of the ground states we use Lions' concentration-compactness method \cite{L}.

Recall the following property of the lambda constant under the Cheeger-Gromov convergence of Riemannian manifolds. Given an n-dimensional complete Riemannian manifold $(M,g)$ with Ricci curvature bounded from below and assume that $V$ is a smooth function bounded from below on $M$ . For $u\in H^1:=H^1(M,g)$, we define $$
I_0(u)=\int_M(4|\nabla u|^2+Vu^2)dv_g.
$$
The lambda constant $\lambda_V(M,g)$ is defined by $$
\lambda_V(M,g)=\inf_{\{u\in H^1; \int_Mu^2dv_g=1\}} I_0(u).$$
From this very definition we can see the below. Suppose that $(M_j^n, g_j, x_j)\to (M_\infty^n, g_\infty, x_\infty)$ in the $C^\infty$ Cheeger-Gromov sense and with $V_j\to V_\infty$ locally uniformly. Then we have
$$
\lambda_V(M_\infty^n, g_\infty)\geq \overline{\lim}_{j\to\infty} \lambda_V(M_j^n, g_j).
$$

 We remark that the assumption that $(M_j^n, g_j, x_j)\to (M_\infty^n, g_\infty, x_\infty)$ in the $C^\infty$ Cheeger-Gromov sense is highly non-trivial and it is not so easy to verify without the hard analysis estimates about curvatures. We are interested in the existence of ground states of lambda constant in complete non-compact Riemannian manifolds. This is a principal eigenvalue problem and our understanding to it is still limited. Motivated by Lions' concentration-compactness method, we can define the lambda constant at infinity $\lambda_\infty$ (see section \ref{part2}). Roughly speaking, we can show in section \ref{part2} that if we have the strict inequality $\lambda(M,g)<\lambda_\infty$, then there is a positive function $u\in H^1$ such that
$I_0(u)=\lambda$. Our results sharpens previous results. When $(M,g)$ is an ALE manifold with $V=s$ the scalar curvature, N.Hirano \cite{Hi} gave a condition about the scalar curvature on the ALE manifold $(M,g)$, which satisfies  $d<d_\infty$. Zhang \cite{Z} studied this problem when $(M,g)$ has bounded geometry. As an application Zhang can prove no breathers theorem for some noncompact Ricci flows \cite{Z2}\cite{H82}.  In this kind of manifolds, the $L^2$ Sobolev inequality is true, which is a key tool in the study of nonlinear problems in Riemannian geometry. We also obtain the ground states for  Perelman's $\mu$-constant on the complete Riemannian manifold with Ricci curvature bounded from below and with positive injectivity radius. It is clear that our goal here is to consider this kind of problems from another angle.

The plan of this paper is below. In section \ref{part1}, we recall the definitions of Perelman's F-functional, the modified scalar curvature, and the lambda constant. We consider the estimation of lambda constant and the existence of the ground states of it in section \ref{part2}. In section \ref{part3}, we introduce another constant $d(M,g)$ on the Nehari manifold and discuss the existence of the ground states of this minimization problem.

\section{Perelman's modified scalar curvature, lambda constant, and Lions lemma}\label{part1}

 We now recall some background about the F-functional and the generalized lambda constant.

Let (M,g) be a n-dimensional Riemannian manifold. Let $V$ be a smooth function bounded from below in $M$. Given a smooth function $f$ with $\int_M e^{-f}dv_g=1$. Set $dm=e^{-f}dv_g$. Similar to G.Perelman \cite{P} one defines the modified scalar curvature related to the measure $dm$ by
$V^m=2\Delta f-|\nabla f|^2+V$ and the $F$-functional
$$
\mathbb{F}(g,f)=\int_M V^mdm.
$$
The properties of this functional are very useful in the understanding of the properties of $W$-functional.

When $V=R$ the scalar curvature of the metric $g$, one natural question related to the Yamabe problem is to find a smooth function $f$ with $\int_M e^{-f}dv_g=1$ such that $R^m$ is a constant. This problem is easy to solve when $M$ is compact.

Assume at this moment $M$ is compact. Using the integration by part, one has
$$
\mathbb{F}(g,f)=\int_M(V+|\nabla f|^2)dm,
$$
which is the original definition of the $F$-functional. The lambda constant is defined by
$$
\lambda (g)=\inf\{\mathbb{F}(g,f); f\in C^\infty(M),\ \  \int_Me^{-f}dv_g=1\}.
$$
Let $u=e^{-f/2}$ and
$$
I_0(u)=\int_M(4|\nabla u|^2+Vu^2)dv_g.
$$
Note that $\int_Mu^2dv_g=1$. Clearly, we have the minimization problem:
\begin{equation}\label{P}
\lambda (g)=\inf\{I_0(u); u\in C^\infty_0(M),\ u>0,  \ \ \int_Mu^2dv_g=1\}
\end{equation}
and one always has a positive minimizer $u$ on $M$ by the direct method. In this case, we have
$\int_Mu^2dv_g=1$ and
$$
-4\Delta u+Vu=\lambda (g)u.
$$
Let $f=-\log u$. By direct computation we have
$$
\lambda(g)=2\Delta f-|\nabla f|^2+V=V^m
$$
with $dm=e^{-f}dv_g=u^2dv_g$.

The minimization problem $(\ref{P})$ is nontrivial when the Riemannian manifold $(M,g)$ is complete and non-compact. One purpose of this paper is to present some consideration of this topic.
We always assume that $(M,g)$ is a complete non-compact manifold with its Ricci curvature bounded from below by some real constant $K$ and with uniform lower bound of the injectivity radius. In this case we have the $L^2$ Sobolev inequality \cite{He}, which says that there is a constant $C>0$ such that
$$
C(\int_M|u|^{p+1})^{2/(p+1)}\leq \int_M(|\nabla u|^2+u^2)
$$
for all $u\in C^\infty_0(M)$ and $p=\frac{n+2}{n-2}$ when $n\geq 3$. From the very definition of $\lambda(g)$, we have
$\lambda (g)\geq \inf_M V(x)$. Hence, the minimization problem $(\ref{P})$ is meaningful. With the help of the lower bound of the Ricci curvature and the $L^2$ Sobolev inequality, we can set-up the Lions lemma as follows.

\begin{Lem}\label{lions}
Assume that $(u_j)\subset H^1$ is a bounded sequence satisfying
$$
\lim_{j\to\infty}\sup_{z\in M}\int_{B_r(z)} |u_j|^2dv=0
$$
for some $r>0$, where $B_r(z)$ denotes the open geodesic ball of radius $r$ centered at $z\in M$. Then
$u_j\to 0$ strongly in $L^q(M,g)$ for all $2<q<2^*:=\frac{2n}{(n-2)_+}$.
\end{Lem}
 This lemma will be used in section \ref{part3} and its proof is given in \cite{ma2}.

We shall denote by $C$ the various uniform constants which may change from line to line.

\section{property of $\lambda$-constant}\label{part2}

We now introduce the scalar curvature potential function on $(M,g)$ with bounded geometry.
Assume $f:M\to R$ be a smooth function on a complete non-parabolic manifold $(M,g)$ with bounded geometry. Denote $G(x,y)$ the minimal Green function on $M$ and let $s=R$ be the scalar curvature of $g$ in this section. Assume that $(M,g)$ has positive injectivity radius and has its Ricci curvature bounded from below by $-K$, where $K\geq 0$ is a constant and assume that $f$ is bounded and in $L^1$. Then the function defined by
$$
u(x)=\int_M G(x,y)f(y)dv_g
$$
satisfies that
$$
-\Delta u=f, \ \ \ on \ \ M.
$$
Furthermore, by the Moser iteration argument and $L^p$ theory of uniform elliptic equations (see also the proof of Lemma b.3 in the appendix of \cite{S}), $u$ has a uniform bound of its gradient on $M$. When $f=V$ which is given before, we call $u$ the potential function of the function $V$.

 We can prove the following result even for more general Riemannian manifolds (see \cite{ma2} for related result).
\begin{Thm}\label{bound} Assume that $(M,g)$ is a complete non-compact manifold with bounded Ricci curvature and with polynomial volume growth. Suppose that the function $V$ has a potential function with uniformly bounded gradient. Then $\lambda_V (g)\leq 0$, and furthermore, when $V\geq 0$, we have $\lambda_V(g)=0$.
\end{Thm}
\begin{proof} By the definition about the polynomial volume growth, we means
that there exist constants $C>0$ and $k\geq 1$, and fixed point $p\in M$ such that the volume $V_p(R)$ of the geodesic ball $B_R(p)$ satisfy
$$
V_p(R)\leq CR^k,  \ \ R>0.
$$

By assumption we have
$V(x)=-\Delta u(x)$ for a smooth function on $M$ and $u$ has uniform bounded gradienti.e., $|\nabla u\leq C$ for some uniform constant $C$.

Assume $\lambda_V (g)>0$. Take $R>1$ and take $\phi(x)$ be the cut-off function on $B_{2R}(p)$ such that $\phi=1$ on $B_R(p)$. By the definition of $\lambda (g)$ we know that
\begin{equation}\label{lambda}
\lambda (g)\int \phi^2\leq \int_M(4|\nabla \phi|^2+V\phi^2).
\end{equation}
Note that
$$
\int_Ms\phi^2=\int_M2\phi\nabla u\cdot \nabla \phi
$$
which is bounded by $CR^{-1}V_p(2R)\leq CR^{k-1}$.
By (\ref{lambda}) we have
$$
\lambda (g)V_p(R)\leq C(R^{k-2}+R^{k-1})\leq CR^{k-1}.
$$
We iterate this relation $k$ times to get
$$
V_p(R)\leq CR^{-1}\to 0
$$
as $R\to \infty$. This is impossible. Thus, we have $\lambda_V (g)\leq 0$.

When $V\geq 0$, by definition, we have $\lambda_V(g)\geq 0$ and then $\lambda_V(g)=0$.
\end{proof}
We remark that in the above argument, we need only assume that $V=div (Y)$ for some smooth vector field in $L^p(M)$ with some $1<p\leq \infty$.

Using similar argument we can prove the below.
\begin{Thm}\label{bound2} Assume that $(M,g)$ is a complete non-compact Riemannian manifold with $V\geq 0$ and quadratic polynomial volume growth. If $\lambda_V(g)=0$, then  $V=0$.
\end{Thm}
\begin{proof} By the property about the quadratic polynomial volume growth, we know
that there exists constants $C>0$ such that the volume $V_p(R)$ of the geodesic ball $B_R(p)$ satisfy
$$
V_p(R)\leq CR^2.
$$
Using the domain exhaustion and the direct method, we can find a positive function $u$ on $M$ such that
$$
-4\Delta u+Vu=\lambda(g)u=0.
$$

Let $w=\log u$ and $q=-V$. Then we have
$$
-\Delta w=q+|\nabla w|^2.
$$

Then for a compactly supported function $\phi$ we have
$$
\int_M(q+|\nabla w|^2)\phi^2=-2\int \phi\nabla\phi\cdot \nabla w
$$
and the right side is bounded by
$$
\epsilon \int_M \phi^2|\nabla w|^2+\epsilon^{-1}\int_M |\nabla \phi|^2
$$
for $\epsilon=1/2$. Then we have
$$
\int_M(q+\frac{1}{2}|\nabla w|^2)\phi^2\leq 2\int_M |\nabla \phi|^2.
$$

Let $r=d(x,p)$ be the distance function and let $\phi=1$ for $r\leq \sqrt{R}$, $\phi=0$ for $r>R$, and
$\phi=2-2\log r/\log R$ for $\sqrt{R}<r\leq \leq R$. By a direct computation we know from the quadratic polynomial volume growth that
$$
\int_M |\nabla \phi|^2\leq \frac{C}{\log R}\to 0
$$
as $R\to \infty$.
Hence we get
$$
\int_{B_{\sqrt{R}}(p)}(q+\frac{1}{2}|\nabla w|^2)\to 0
$$
as $R\to\infty$. Hence $q=0$ and $|\nabla w|^2=0$ on $M$, which implies that $V=0$.

\end{proof}

One may refer to \cite{MW} and \cite{WW} for more interesting results about volume growth estimates in Riemannian manifolds with densities.

As for the existence of minimizers of the minimization problem $(\ref{P})$ is nontrivial when the Riemannian manifold $(M,g)$ is complete and non-compact. We have the following result by using P.L.Lions' concentration-compactness method (see \cite{L} p.115 ff).
Recall $$
I_0(u,g):=\int_M(4|\nabla u|^2+Vu^2)dv
$$
on the manifold
$$
\Sigma:=\{u\in H^1(M,g); \ \ \int_Mu^2dv=1\}.
$$
Fix $0\in M$. Assume that $\lim_{d(x,0)\to\infty} V(x)=V_\infty$ for some real constant $V_\infty$. We defines the lambda constant of $(M,g)$ at infinity as the quantity
$$
\lambda_\infty=\inf\{I_0^\infty(u,g); u\in C_0^{\infty}(M), |u|_2=1\},
$$
where
$$
I_0^\infty(u,g)=\int_M(4|\nabla u|^2+V_\infty u^2)dv.
$$

\begin{Thm}\label{lambda2} Assume that $(M,g)$ is a complete non-compact Riemannian manifold with its Ricci curvature bounded from below. Assume that $\lim_{d(x,0)\to\infty} V(x)=V_\infty$ for some real constant $V_\infty$ and assume that $\lambda(g)<\lambda_\infty$, then there is a positive function $u\in H^1$ such that
$I_0(u,g)=\lambda_V(g)$.
\end{Thm}
Here is the proof of Theorem \ref{lambda2}. We simply write $\lambda(g)=\lambda_v(g)$ when there is no confusion.
 \begin{proof}
 Let $(u_j)\subset \Sigma$ be the minimizing sequence of the minimization problem $(\ref{P})$. Then it is easy to see that $(u_j)$ is a bounded sequence in $H^1$. Then we may assume that $(u_j)$ converges weakly in $H^1$ to a function $u\in H^1$ and converges in $L_{loc}^2(M,g)$.
Let $v_j=u_j-u$. Then we have
$$I_0(u_j)=I_0(v_j)+I_0(u)+\circ(1)\to \lambda(g)
$$
and
$$
1=\int_Mu_j^2=\int_Mv_j^2+\int_M u^2+\circ(1)
$$
We also have
$$
I_0(v_j)=I_0^\infty(v_j)+\circ(1).
$$
Hence.
$$
I_0^\infty(v_j)+I_0(u)+\circ(1)\to \lambda(g).
$$

Let $m=\int_Mu^2$. Then $m\in [0,1]$.
If $m=0$, then we have $u_j=v_j$ and
$$
\int_{B_R(0)} u_j^2=\circ(1)
$$
for any fixed $R>1$.
Then
$$
\lambda_\infty\leq I_0^\infty(v_j,g)\to \lambda(g),
$$
which is a contradiction to the assumption $\lambda(g)<\lambda_\infty$.

We now have $m>0$. If $m<1$.
Then $\int_Mv_j^2\to 1-m>0$.
Then we have
$$
\lambda(g)\leq I_0(\frac{u}{\sqrt{m}})=\frac{I_0(u)}{m}
$$
and
$$
\lambda_\infty(g)\leq I_0^\infty(\frac{v_j}{\sqrt{1-m}})=\frac{\lambda(g)-I_0(u)}{1-m}
$$
Hence we have
$$
\lambda(g)=I_0(u)+(\lambda(g)-I_0(u))\geq m\lambda(g)+(1-m)\lambda_\infty(g),
$$
which implies a contrary conclusion again that
$$
\lambda(g)\geq \lambda_\infty(g).
$$
Hence,
$\int_Mu^2=1$ and the minimizing sequence $(u_j)$ converges strongly in
$H^1$ to the limit $u$.
\end{proof}

In general, we can define
$$
\lambda_\infty=\lim_{r\to\infty}\inf\{\int_{M-B_r(0)}(4|\nabla u|^2+V u^2) dv; u\in C_0^{\infty}(M-B_r(0));u\not=0, \int_Mu^2dv=1\}.
$$
We have the following result.

\begin{Thm}\label{lambda3} Assume that $(M,g)$ is a complete non-compact Riemannian manifold with its Ricci curvature bounded from below. Assume  that $\lambda_V(g)<\lambda_\infty$, then there is a positive function $u\in H^1$ such that
$I_0(u,g)=\lambda(g)$.
\end{Thm}

\begin{proof}
 Let $(u_j)\subset \Sigma$ be the minimizing sequence of the minimization problem $(\ref{P})$.
  Without loss of generality, we may assume that $V\geq 1$ in $M$. Then it is easy to see that $(u_j)$ is a bounded sequence in $H^1$, and we may assume that $(u_j)$ converges to $u$ weakly in $H^1$ and strongly in $L^2_{loc}(M)$. Applying Lions' concentration-compactness method to the measures
 $$
 \rho(u_j)=|\nabla u_j|^2+Vu_j^2.
 $$

In the vanishing case, we have
$$
\overline{\lim}_{j\to\infty}\sup_{y\in M}\int_{B_r(y)}(|\nabla u_j|^2+u_j^2)=0, \ \ for \ all \ r>0.
$$
Fix $y\in M$. Choose the cut-off function $\xi_R$ such that $\xi_R(x)=0$ in $B_r(y)$ and $\xi_r(x)=1$ outside $B_{2r}(y)$. Then we have
$$
\int_M(|\nabla (\xi_ru_j)|^2+(\xi_ru_j)^2)dv=\int_{M}(|\nabla u_j|^2+u_j^2)dv+\circ(1).
$$
Recall that $u_j\in \Sigma$,
$$
I_0(u_j,g)=\int_M(4|\nabla u_j|^2+Vu_j^2)dv\to \lambda.
$$
Since
$$
\int_M(\xi_ru_j)^2=1+\circ(1),
$$
we have
$$
\int_{M-B_r(0)}(4|\nabla (\xi_ru_j)|^2+V (\xi_ru_j)^2) dv\geq \lambda_\infty.
$$
However, we have
$$
\int_{M-B_r(0)}(4|\nabla (\xi_ru_j)|^2+V (\xi_ru_j)^2) dv=\int_{M}(4|\nabla u_j|^2+V u_j^2) dv\to \lambda,
$$
which gives us that $\lambda\geq \lambda_\infty$, a contradiction to our assumption.

In the dichotomy case, we have
\begin{equation}\label{star}
\lambda\leftarrow I_0(u_j,g)=I_0(u^1_j,g)+I_0(u^2_j,g)+\circ(1)
\end{equation}
for $u_j=u_j^1+u_j^2$ for two sequences $u_j^1\in H^1$ and $u_j^2\in H^1$
and
$$dist(supp(u_j^1),supp(u_j^2))\to \infty.$$ We may assume that $dist(y,supp(u_j^2))\to\infty$ and $\int_M(u_j^2)^2\to \alpha>0$. Clearly, $\alpha\leq 1$.
Then we have
$$
I_0(\frac{u^2_j}{\sqrt{\alpha}},g)\geq \lambda_\infty,
$$
that is,
$$
I_0(u^2_j,g)\geq \alpha \lambda_\infty.
$$
Similarly, we have
$$
I_0(u^1_j,g)\geq (1-\alpha) \lambda.
$$
Then we have
$$
I_0(u^1_j,g)+I_0(u^2_j,g)\geq (1-\alpha)\lambda+ \alpha\lambda_\infty.
$$
Combining this with (\ref{star}) we obtain that
$$
\lambda \geq (1-\alpha)\lambda+ \alpha\lambda_\infty,
$$
which gives a contradiction to our assumption that $\lambda<\lambda_\infty$.

So we are left the compactness case. In this case, it is standard to get the $H^1$-convergence of the sequence $(u_j)$ to the function $u$. In fact, in this case, there is a sequence $z_j\in M$, for any $\epsilon>0$, there is a large $r>0$ such that
$$
\int_{B_r(z_j)}(|\nabla u_j|^2+Vu_j^2)\geq \lambda-\epsilon,
$$
where
$$
\lambda=\lim_{j\to\infty} \int_{M}(|\nabla u_j|^2+Vu_j^2)dv.
$$
If $(z_j)\subset M$ is unbounded, then we get again
$\lambda\geq \lambda_\infty$, which is impossible. We are left the case when the sequence $(z_j)\subset M$ is bounded. We can get that $\lim I_0(u_j)=I_0(u)=\lambda$ and the $H^1$-convergence of the sequence $(u_j)$ to the function $u\geq 0$. By the regularity theory we know that $u$ is smooth. By the maximum principle we know that $u>0$ in $M$. This completes the proof of Theorem \ref{lambda3}.

\end{proof}

\section{ground states of W-functional}\label{part3}

For $u\in H^1(M,g)$ with $\int_Mu^2dv=1$, G.Perelman \cite{P} \cite{chow06} defines the W-functional (with the parameter $\tau$ being fixed and normalized) by
$$
W(u,g)=\int_M(4|\nabla u|^2+Vu^2-u^2\log u^2)dv
$$
and the $\mu$-constant by
$$
\mu(g)=\int_M\{W(u,g); \int_Mu^2=1\}.
$$
One can show that $m(g)$ is well-defined. The Euler-Lagrange equation of the $W$-functional is
\begin{equation}\label{EL1}
4\Delta u-Vu+2u\log u+\mu(g)u=0, \ \ M.
\end{equation}
One can define two related functionals to (\ref{EL1}) by
$$
I(u):=I(u,g):=\int_M(4|\nabla u|^2+Vu^2-\frac{1}{2}u^2\log u^2+\frac{u^2}{2})dv
$$
and
$$
N(u)=\int_M(4|\nabla u|^2+Vu^2-\frac{1}{2}u^2\log u^2)dv
$$
where $u\in H^1(M,g)$. Define the Nehari manifold by
$$
\mathbb{N}=\{u\in H^1(M)-\{0\}; N(u)=0\}
$$
and
$$
d(g)=\inf\{I(u); u\in \mathbb{N}\}.
$$
Note that for $u\in \mathbb{N}$, $I(u)=\int_M \frac{u^2}{2}dv$. One show show that $d:=d(g)>0$.
The two constants $\mu(g)$ and $d$ are related by the relation
$\mu(g)=\log (2\sqrt{d})$.
In fact, for $u\in H^1(M)$, $|u|_2=1$, we have
$$
\mu(g)+2\int_M u^2\log |u|dv\leq \int_M(4|\nabla u|^2+Vu^2)dv.
$$
For $f\in H^1(M)-\{0\}$, let $u=f/|f|_2$. Then we have
$$
|f|_2^2(\mu(g)-2\log |f|_2)+2\int_M f^2\log |f|dv\leq \int_M(4|\nabla f|^2+Vf^2)dv.
$$
For $w\in \mathbb{N}$ and $a>0$, putting $f=aw$ we have
$$
|w|_2^2(\mu(g)-2\log |w|_2-2\log a)+2\int_M w^2(\log |w|+\log a)dv\leq \int_M(4|\nabla w|^2+Rw^2)dv.
$$
Using the relation $N(w)=0$, we know that
$$
\mu(g)-2\log |w|_2\leq 0.
$$
Then, $\mu(g)\leq \log(2\sqrt{d})$. Conversely we can also prove that $2\sqrt{d}\leq e^{\mu(g)}$. Hence, we have the following result.

\begin{Pro} Assume that  $(M,g)$ is an n-dimensional  complete non-compact Riemannian manifold. Assume that $V$ is given as before and the constants $d(g)$ and $\mu(g)$ defined above, we have
$$
\mu(g)= \log(2\sqrt{d(g)}).
$$
\end{Pro}

Recall that Zhang \cite{Z} defines the Log-Sobolev constant of $(M,g)$ at infinity as the quantity
$$
\mu_\infty=\lim_{r\to\infty}\inf\{W(u,g); u\in C_0^{\infty}(M-B_r(0)), |u|_2=1\}.
$$
Similarly we can define
$$
d_\infty=\lim_{r\to\infty}\inf\{\int_{M-B_r(0)} \frac{u^2}{2}dv; u\in C_0^{\infty}(M-B_r(0));u\not=0,  N(u)=0\}
$$

Then using Lions' variational principle at infinity \cite{L} we can prove the below.

\begin{Thm}\label{lions} Assume that $(M,g)$ is a complete non-compact Riemannian manifold with its Ricci curvature bounded from below and
$$\inf_{x\in M}vol(B_1(x))>0,$$ where $vol(B_1(x))$ is the volume of the unit ball at $x\in M$.
 Assume that $d<d_\infty$. Then $d$ is attained at some $u\in \mathbb{N}$ such that
$$
-4\Delta u+Vu=u\log u
$$
and $u>0$ such that $d=\frac{1}{2}\int_M u^2$.
\end{Thm}
We remark that in our case the $L^2$-Sobolev inequality in $(M,g)$ is true \cite{He}.

The argument is similar to that of P.L.Lions \cite{L} p.115 ff. and Theorem 4.3 in \cite{S}. Here we choose $\rho(u)=|\nabla u|^2+u^2$. Since the proof is standard, we omit the detail of the proof.

\end{document}